\documentclass{amsart}
\usepackage{graphicx} % Required for inserting images
\usepackage[T1]{fontenc}
\usepackage{tikz-cd}
\usepackage{amsmath}
\usepackage{calligra}
\usepackage{amssymb}
\usepackage{amsthm}
\usepackage{xcolor}
\usepackage[margin=1.2in]{geometry}
\usepackage{hyperref}
\usepackage{mathrsfs} 
\usepackage{yfonts}
\usepackage{egothic}

\usepackage{amssymb}
\usepackage{amsthm}
\usepackage{anyfontsize}
\usepackage{geometry} % see geometry.pdf on how to lay out the page. There's lots.
\usepackage{tikz-cd}
\usepackage{amsbsy}
\usepackage{amssymb,latexsym} 
\usepackage{amsmath,amsthm} 
\usepackage[english,]{babel}
\usetikzlibrary{cd}
 \usepackage{mathtools}
 \usepackage{hyperref}
\usepackage{mathrsfs}
   \newcommand{\ft}{\mathfrak{t}}

\newcommand{\cF}{\mathcal{F}}

\newcommand{\K}{\mathbb{K}}

\newcommand{\fg}{\mathfrak{g}}

\newcommand{\R}{\mathbb{R}}

 \newcommand{\cC}{\mathbb{C}}

\newcommand{\oO}{\mathbb{O}}
\newcommand{\hH}{\mathbb{H}}

\renewcommand\labelenumi{(\roman{enumi})}
\renewcommand\theenumi\labelenumi

\newtheorem{thm}{Theorem}

\newtheorem{prop}{Proposition}

\theoremstyle{definition}
\newtheorem{dfn}{Definition}

\theoremstyle{remark}
\newtheorem{rem}{Remark}
\newtheorem{ex}{Example}
\theoremstyle{remark}

\title{Wishart cones and quantum geometry}
\author{Noemie C. Combe}
\address{n.combe@uw.edu.pl, Ulica Banacha 2, University of Warsaw}
\date{December, 2024}

\begin{document}

\begin{abstract}
  An important object appearing in the framework of the Tomita--Takesaki theory is an invariant cone under the modular automorphism group of von Neumann algebras. As a result of the connection between von Neumann algebras and quantum field theory, von Neumann algebras have become increasingly important for (higher) category theory and topology.

 We show explicitly how an example of a class of cones discovered by Connes--Araki--Haagerup (CAH), invariant under the modular automorphism group, are related to Wishart laws and information geometry. Given its relation to 2D quantum field theory this highlights new relations between (quantum) information geometry and quantum geometry.    
\end{abstract}
\maketitle
\section{Introduction}
 An important object appearing in the framework of the Tomita-Takesaki theory~\cite{Ta70} is an invariant cone under the modular automorphism group of von Neumann algebras. As a result of the connection between von Neumann algebras and quantum field theory, von Neumann algebras have become increasingly important for (higher) category theory and topology.

 We show explicitly that an example of a class of cones discovered by Connes--Araki--Haagerup (CAH)~\cite{Connes}, which are invariant under the modular automorphism group, are related to Wishart laws and information geometry. This highlights new relations between (quantum) information geometry and quantum geometry reinforcing results in \cite{CoMa,CMM,CCN1,CCN2} on information geometry (flat manifolds of exponential type) and 2D quantum field theory. 

\, 

 The CAH invariant cones are self-dual and when those cones are finite dimensional, they are transitively homogeneous. In the finite dimensional case, the CAH cones are in bijection with the class of formally real Jordan algebras. There are five irreducible such algebras. Three of them correspond to the set of self adjoints $n\times n$ matrices with elements in the field  $\R, \cC$ and $\hH$. One is the exceptional algebra of  self adjoints $3\times 3$ matrices with coefficients in the field of octonions. The last class is of a different nature as it corresponds to the spin factor algebra. 
 
In the finite dimensional case,  we have proved in \cite{C23,C24} that those cones satisfy the axioms of a pre-Frobenius domain and that they contain a subspace being a Frobenius manifold, see \cite[p.19]{Man99} for a definition on Frobenius manifolds and pre-Frobenius structures.  

\, 

In classical literature, recall that by definition the Wishart law is characterized by its probability density function, which is as follows. Let $\mathbf {X} $ be an $m \times m$ symmetric matrix of random variables, positive semi-definite. Let $\mathbf {B} $ be a (fixed) symmetric positive definite matrix of size $m \times m$.

Then, if $n \geq m$, $\mathbf{X}$ has a Wishart distribution with $n$ degrees of freedom if it has the following probability density function

\[{f_{\mathbf{X} }(\mathbf{X} )={\frac {1}{2^{nm/2}\left|{\mathbf {B} }\right|^{n/2}\Gamma _{m}\left({\frac {n}{2}}\right)}}{\left|\mathbf {X} \right|}^{(n-m-1)/2}e^{-{\frac {1}{2}}\operatorname {tr} ({\mathbf {B} }^{-1}\mathbf {X} )}}\]
where ${\left|{\mathbf {X} }\right|}$ is the determinant of ${\mathbf{X}}$ and $ \Gamma_m$ is the multivariate gamma function defined as

\[{ \Gamma _{m}\left({\frac {n}{2}}\right)=\pi^{m(m-1)/4}\prod _{j=1}^{m}\Gamma \left({\frac {n}{2}}-{\frac {j-1}{2}}\right).}\]

{{\bf Acknowledgements}
I would like to thank G\'erard Letac for discussions on Wishart laws and thank Jacques Faraut for discussions on symmetric convex cones. This research is part of the project No. 2022/47/P/ST1/01177 co-founded by the National Science Centre  and the European Union's Horizon 2020 research and innovation program, under the Marie Sklodowska Curie grant agreement No. 945339 \includegraphics[width=1cm, height=0.5cm]{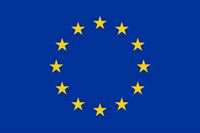}.}

\section{The geometry of symmetric cones}

\subsection{Strictly convex cones} 
In the following parts of this article we always consider {\it strictly convex cones}. Note that for brevity we simply refer to them as {\it convex cones.} 

Let us recall some elementary notions on strictly convex cones (see \cite{FK} for further information). 

\begin{dfn}
Let $V$ be a finite dimensional real vector space. Let $\langle-,-\rangle$ be a non-singular symmetric bilinear form on $V$.
A subset $\Omega \subset V$ is a convex cone if and only if $x,y \in \Omega$ and $\lambda,\mu >0$ imply $\lambda x+\mu y \in \Omega$.
\end{dfn} 

\subsection{Homogeneous cones} 
The automorphism group $G(\Omega)$ of an open convex cone $\Omega$ is defined by 
\[G(\Omega)=\{g\in GL(V)\, |\, g\Omega=\Omega\}\]
An element $g\in GL(V)$ belongs to $G(\Omega)$ iff $g\overline{\Omega}=\overline{\Omega}$ \cite{FK}
So, $G(\Omega)$ is a closed subgroup of $GL(V)$ and forms a Lie group. 
The cone $\Omega$ is said to be {\it homogeneous} if $G(\Omega)$ acts transitively upon $\Omega$.

\subsection{Symmetric cones} 
From homogeneous cones one can construct symmetric convex cones. Let us introduce the definition of an open dual cone. An open dual cone $\Omega^*$ of an open convex cone is defined by $\Omega^*=\{y\in V\, |\, \langle x,y \rangle>0,\, \forall\, x\in \overline{\Omega}\setminus 0 \}$. A homogeneous convex cone $\Omega$ is symmetric if $\Omega$ is self-dual i.e. $\Omega^*=\Omega$. Note that if $\Omega$ is homogeneous then so is $\Omega^*$.

\subsection{Automorphism group}  Let us go back to the automorphism group of $\Omega$. This discussion relies on Prop I.1.8 and Prop. I.1.9 in \cite{FK}.

\smallskip 

Let $\Omega$ be a symmetric cone in $V$.  For any point $a\in \Omega$ the stabilizer of $a$ in $G(\Omega)$ is given by 
\[G_a=\{g\in G(\Omega)\, |\, ga=a\}.\]

By [Prop I.1.8 \cite{FK} ], if $\Omega$ is a proper open homogeneous convex cone then for any $a$ in $\Omega$, $G_a$ is compact. Now, if $H$ is a compact subgroup of $G$ then $H\subset G_a$ for some $a$ in $\Omega$. This means that the groups $G_a$ are all maximal compact subgroups of $G$ and that if $\Omega$ is homogeneous then all these subgroups are isomorphic. 

By \cite[Prop. I.1.9]{FK}, if $\Omega$ is a symmetric cone, there exist points $e$ in $\Omega$ such that $G(\Omega)\cap O(V)\subset G_e$, where $O(V)$ is the orthogonal group of $V$. For every such $e$ one has $G_e=G\cap O(V)$ 

Suppose $\Omega$ is a convex homogeneous domain in $V$. Assume that
\begin{itemize}
   \item   $G(\Omega)$ is the group of all automorphisms;
   \item   $G_e=K(\Omega)$ is the stability subgroup for some point $x_0\in \Omega$;
   \item   $T(\Omega)$ is a maximal connected triangular subgroup of $G(\Omega).$ 
\end{itemize}

This decomposition on the Lie group side leads naturally to its Lie algebra. Cartan's decomposition for the Lie algebra tells us that $\fg=\mathfrak{k} \oplus\ft,$ 

where:

\begin{itemize}
   \item   $\mathfrak{t}$ can be identified with the tangent space of $\Omega$ at $e$. 

   \item  $\mathfrak{k}$ is the Lie algebra associated to $K(\Omega)$
\end{itemize}
 and
\[[\ft,\ft]\subset \mathfrak{k},\]
\[[\mathfrak{k},\ft]\subset \ft.\]

\subsection{Classification of cones}
Any symmetric cone (i.e. homogeneous and self-dual) $\Omega$ is in a unique way isomorphic to the direct product of irreducible symmetric cones $\Omega_i$ (cf. \cite[Prop. III.4.5]{FK}). According to Vinberg we have:

\begin{prop}~\label{P:Vclass}
Each irreducible homogeneous self--dual cone belongs to one
of the following classes:
\vspace{3pt}\begin{table}[ht]
    \centering
    \begin{tabular}{|c|c|c|}
  \hline
Nb & Symbol & Irreducible symmetric cones \\
 \hline
 1. &      $ \mathscr{P}_n(\R)$ &  Cone of $n \times n$ positive definite symmetric real matrices. \\
     &  & \\
      
    2.  &      $ \mathscr{P}_n(\cC)$ &  Cone of $n \times n$ positive definite self-adjoint complex matrices. \\
  & &\\
   %\hline
       3.  &    $ \mathscr{P}_n(\hH)$ &  Cone of $n \times n$ positive definite self-adjoint quaternionic matrices. \\
           &   & \\
           % \hline
        4. &   $ \mathscr{P}_3(\oO)$ & Cone of $3 \times 3$  positive definite self-adjoint octavic matrices. \\
           &   & \\
          %  \hline
    5. &    $\Lambda_n$    & Lorentz cone  given by $x_0>\sqrt{\sum_{i=1}^n x_i^2}$ (aka spherical cone). \\
        &   & \\
       \hline
    \end{tabular}
    \caption{Classification of irreducible symmetric cones}
    \label{tab:cones}
\end{table}
\end{prop}

\begin{rem}
We have two remarks. The first is that $\mathscr{P}_3(\oO)$ corresponds to the Cayley algebra. 
The second is that  spherical cone $\Lambda_n$ corresponds here to an $n$-dimensional Anti-de-Sitter (AdS) space.
\end{rem}

 \subsection{Jordan algebra structures}\label{S:JordanList}
Recall the tight relations between those cones and algebraic objects, namely formally real simple Jordan algebras. 
We introduce some notations: 
\begin{itemize}
   \item[---]   $Sym(n,\mathbb{K})$ denotes the space of symmetric matrices of dimension $n\times n$ defined over the field $\mathbb{K}$.

   \item[---]   $Herm(n,\mathbb{K})$ denotes the space of hermitian matrices of dimension $n\times n$ defined over the field $\mathbb{K}$.
\end{itemize}
\begin{dfn}
 An algebra $(\mathscr{A}^+,\circ)$ is a Jordan algebra if:
 \begin{itemize} 
   \item[---]   it is commutative and
   \item[---]   $(x^2\circ y)\circ x = x^2\circ (y\circ x)$, 
 \end{itemize} where we have $x^2=x\circ x$ and $x,y\in \mathscr{A}^+.$ 
 
 A Jordan algebra is called Euclidean (or formally real) if it satisfies the formal reality axiom: \[x_1^2+\cdots+x_n^2 =0\quad \text{implies}\quad x_1=\cdots=x_n=0\] 
 \end{dfn}
 
In particular, a formally real Jordan algebra is semi-simple. To any strictly convex symmetric cone ({\bf SCS}-cone) there exists a bijectively corresponding semisimple formally real Jordan algebra. 

\smallskip 

\begin{ex}
  The tangent space to  $\mathscr{P}_n(\R)$, is $Sym(n,\R)$. Given $X,Y\in Sym(n,\R)$, one gets a new product (giving a Jordan algebra): \[X\circ Y=\frac{1}{2}(XY+YX),\] where $XY$ is the standard matrix product.  
\end{ex}
\section{Measures on homogeneous spaces}
\subsection{Invariant measures on homogeneous spaces}
We are interested in recalling some properties of invariant measures on homogeneous spaces which will naturally apply to the convex cones discussed above. 
Let $\Omega$ be the non-empty cone and $G$ a group. Then $\Omega$ is called a $G$-space if it is equipped with an action of the group $G$ on $\Omega$. Naturally, $G$ acts by automorphisms on the set $\Omega$. So, supposing that $\Omega$ belongs to some category then the elements of the group $G$ are assumed to act as automorphisms in that same category. A homogeneous space is a $G$-space on which $G$ acts transitively.

\, 

Given a coset space $G/H$, this forms a homogeneous space for $G$ with a distinguished point  (the coset of the identity). If the action of $G$ on $\Omega$ is continuous and $\Omega$ is Hausdorff then $H$ is a closed subgroup of $G$. If $G$ is a Lie group then $H$ is a Lie subgroup by Cartan's theorem and $G/H$ is a smooth manifold carrying a unique smooth structure, compatible with the group action. 

\, 

Suppose that $G$ is a locally compact topological group acting on $\Omega$. Then, it acts (on the left) on continuous functions on $\Omega$ by the following formula:
\[(\lambda(g)f)(x)=f(g^{-1}\cdot x),\] where $x\in V$ and $g\in G$;
and on the right by 
\[(\rho(g)f)(x)=f(x\cdot g)\]

\, 

The basic theorem about Haar measures goes as follows. Suppose $G$ is a locally compact topological group. Then, there is a nonzero Borel measure $d^l_G$ on  $G$ with the left invariance property that:

\[\int_G f(g^{-1}x)d^l_G(x)=\int_G f(x)d^l_G(x),\]
for any $g\in G$ and $f$ a continuous function on $G$. Any right translate of $d^l_G$ shares this left-invariance and therefore is a positive scalar multiple of $d^l_G$:
\[\int_G f(xg)d^l_G(x)=\delta_G(g^{-1})\int_G f(x)d^l_G(x),\quad (h\in G, f\in C(G)),\]
where $\delta_G:G\to \mathbb{R}^+$ is a continuous homomorphism being the {\it modular character} of $G$.

\begin{rem}
Whenever $G$ is a Lie group, we have: $\delta_G(g)=|\det (Ad(g))|$. 
\end{rem}
Assume $G$ is a locally compact group and that $H$ is a closed subgroup. Suppose that $d^l_G$ and $d^l_H$ are left Haar measures.
There is a non-zero $G$-invariant Borel measure on $G/H$ if and only if the modular functions satisfy $\delta_H=\delta_G|_{H}$.
If the modular functions agree, then the invariant measure on $G/H$ is unique up to  a positive multiple. It may be normalized so that the following equality is true: 
\[\int_Gf(x)d^l_G(x)=\int_{G/H}\left[\int_{H}f(xh)d^l_H(h) \right] d_{G/H}(xH),\] for $f$ a continuous function on $G$.

\subsection{Multipliers}
Let us go back to the setting given by the SCS cones. Assume $\Omega$ is a strictly convex symmetric cone of finite dimension with enveloping vector space $V$. 
 
 \,
 
As previously, let $G$ be the Lie group acting on $\Omega$. It is given by 
 \[G\times \Omega\to \Omega\]
 \[(g,p)\mapsto g(p):=gp.\]
 This mapping is differentiable and induces an action of $G$ on the set of measures on $\Omega$. It is the connected component of the identity in the automorphism group $Aut(\Omega)$. The action is transitive. 
 \,

 Let $\xi: G\to \mathbb{R}_+$ be the multiplier on the group $G$.
 The group $G$ acts on $\Omega$ transitively and properly. For each multiplier $\xi$ there exists one and only one relatively invariant measure $\mu^\xi$ on the open subset $\Omega\subset V$ under the action of $G$. In other words, for any $g\in G$, we have: $$g^{-1}\cdot \mu^{\xi}=\xi(g)\, \mu^{\xi}.$$

\,

The following integral converges:
$\Psi^\xi(\sigma)=\int_{\Omega}\exp{\{-\sigma(s)\}}d\mu^{\xi}(s)$ where $\sigma\in \Omega^*$ and has the following property $$\Psi^{\xi}(^tg^{-1}\, \sigma)=\xi(g)\, \chi^{\xi}(\sigma),$$ where $g\in G, \sigma\in \Omega^*$. 

For any multiplier $\xi$ there exists a  positive function $n^{\xi}:\Omega\to \R_{+}$ satisfying the property 
\[n^\xi(gs)=\xi(g)n^\xi(s),\]
$g\in G,\, s\in \Omega$. Such a function is unique up to multiplication by a positive constant since the action of $G$ on $\Omega$ is transitive.  

\subsection{Generalized Wishart laws} The previous discussion allows us now to consider generalized Wishart laws.
Let $\Omega$ be a SCS cone with enveloping vector space $V$; $\xi$ is a multiplier on the Lie group and $\mu^\xi$ the relatively invariant measure on the cone $\Omega$ with multiplier $\xi$ with respect to the action of $G$ on $\Omega$.
The generalized Wishart probability distribution $W_{\sigma,\xi}$ defined on the cone $\Omega$ is given by:
\[dW_{\sigma,\xi}=\frac{1}{n^\xi(\sigma)}\exp\{-\sigma^{-\xi}(s)\}d\mu^{\xi}(s),\]

where $$n^\xi(\sigma)=\int_{\Omega}\exp\{-\sigma^{-\xi}(s)\}d\mu^\xi(s),$$ $\sigma\in \Omega^*$, $s\in\Omega$, $\xi$ is the multiplier. Note that the Wishart distribution only depends on $ \mu^{\xi}$ through $\xi$. In particular any other relatively invariant measure with given multiplier $\xi$ yields the same Wishart distribution. If $W_{\sigma_1,\xi_1}= W_{\sigma_2,\xi_2}$ then $(\sigma_1,\xi_1)=(\sigma_2,\xi_2)$. 

\, 

The measure $d\mu(s):=n^\xi(s)^{-1}d\mu^\xi(s)$ is invariant under the action of $G$ on $\Omega$ and independent on the choice of the relatively invariant measure 
$\mu^{\chi}$. 

\subsection{Wishart laws on cones}
An example of such cones is the space of positive definite symmetric matrices over a real division algebra $\K$ (see Table\ref{tab:cones}). The real division algebra includes real numbers $\R$, complex numbers $\cC$, quaternions $\hH$ and octonions $\oO$. Note that for the octonion case, one considers the space of $3\times 3$ matrices. 

%\, 

For simplicity, let us consider the cone $\mathscr{P}_n(\R)$ for $\K=\R$. In a multivariate sample, the classical Wishart distribution arises as the distribution of the maximum likelihood (ML) estimator of the covariance matrix. Assume $I$ is a finite set. We shall denote $|I|=n$ the cardinality of the set $I$. Let $S$ be a positive definite symmetric matrix of random variables, being a point of the SCS cone of size $n\times n$. We take a fixed matrix $\Sigma$ symmetric positive definite matrix of size $n\times n$, being the expectation. If we take $m\geq n$, $S$ has a Wishart distribution with $m$ degrees of freedom.  \begin{rem}
     Note that this last condition insures that the maximum likelihood of $\Sigma$ is unbiased.  
 \end{rem}

\, 

The classical Wishart distribution on this cone is given by  
\[dW_{\Sigma,\lambda}(S)=\frac{\lambda^{n\lambda}\det(S)^{\lambda-\frac{n+1}{2}}}{\pi^{\frac{n(n-1)}{4}} \prod\limits_{i=1}^{n}(\Gamma(\lambda-\frac{i-1}{2}))\det(\Sigma)^\lambda }\exp\{-\lambda tr(\Sigma^{-1}S)\}dS,\]
such that $\lambda$ is the shape parameter with $\lambda:=\frac{1}{2}m\in \{\frac{n}{2},\frac{n+1}{2},\frac{n+2}{2},\cdots\}$ ($m$ is the degree of freedom) and $\Sigma=m{\bf B}$ is the expectation, being an element of the cone of positive definite symmetric matrices (${\bf B}$ is the multivariate scale);
 $dS$ denotes the standard Lebesgue measure on the cone; $\det$ is the determinant, $tr$ is the trace. Remark that one has a probability measure for $\lambda>\frac{n-1}{2}$.

\, 

The family of Wishart distributions on the sample space given by the cone $\mathscr{P}_n(\K)$ (or $\Lambda_n$) with a fixed shape parameter $\lambda$ constitutes a statistical model: the classical Wishart model. 

\section{Monoidal structures}\label{S:Wish}
We highlight that we have a very convenient structure occurring: a monoidal structure. This is important in the sense that any SCS cone can be a linear combination of irreducible SCS cones. This comes from a more global structure arising from categories: symmetric monoidal categories and impacts the Wishart laws too. 

\subsection{Monoidal  (= tensor) categories}   {\it Data:} 
 multiplication $\otimes$ of objects, with an identity object ${\bf 1}$ and natural isomorphisms
 $$
\alpha_{A,B;C}: (A\otimes B)\otimes C  \to A\otimes (B\otimes C), \quad
\tilde{\rho}_A: A\otimes {\bf 1} \to A, \quad
\tilde{\lambda}_A: {\bf 1}\otimes A \to A .
$$

\subsection{Symmetric monoidal categories.} Additional {\it twist} isomorphisms
$\tau_{A,B}: A\otimes B \to B\otimes A$, with $\tau_{A,B}\tau_{B,A} = {id}_{A\otimes B}$,
plus many commutative diagrams.

\subsection{Symmetric monoidal category CAP}
We introduce the category $CAP$~\cite{Ch65}. 
\begin{dfn}
 The category $CAP$ of probability distributions consists of the following data:
 \vspace{3pt}
 
 \begin{enumerate}
\item An object of $CAP$ is the set $Cap (X , \cF )$ of all probability distributions 
on a $\sigma$--algebra $(X , \cF)$.
 \vspace{3pt}
 
\item One (Markov) morphism $\Pi \in Hom_{CAP} (Cap (X_1 , \cF_1 ), Cap (X_2 , \cF_2 ))$.
It is given by a ``transition measure'', i.e. a function $\Pi \{* | x^{\prime}\}$ upon $\cF_2\times X_1$
such that for a fixed $U\in \cF_2$, $\Pi \{U | x_1\}$ is $\cF_1$--measurable function
on $X_1$, and for a fixed $x_1 \in X_1$,  
$\Pi \{U | x_1\}$ is a probability distribution upon $\cF_2$.
 \vspace{3pt}
 
Explicitly, such $\Pi$ sends the probability distribution $P_1\in Cap(X_1, \cF_1)$
to the probability distribution $P_2\in Cap(X_2, \cF_2)$ given by
 \[
P_2(X_2 | x_1) := \int_{X_1} \Pi\{ * | x_1\} P_1\{dx_1\}.
\]
\end{enumerate}
\end{dfn}
A direct product of measurable spaces and the corresponding tensor product of all the collections of probability measures  can be defined. This multiplication is functorial with respect to the Markov category. This forms a symmetric monoidal category.
We apply this to the case of Wishart laws.
\subsection{Monoidal Wishart laws}
The classical Wishart distribution are concentrated on open cones of symmetric positive defined matrices of size $n$. Those cone do not need to be irreducible and this is where the monoidal structure comes in. 

Recall from Table~\ref{tab:cones} that there are five irreducible cones defined respectively over a real division algebra i.e. $\R$, $\cC$, $\hH,$ and $\oO$ (with an exceptional cone of $3\times 3$ positive definite hermitian matrices with entries in octonions $\oO$) and an extra Anti-de-Sitter cone of Lorentz type, endowed with a split-complex geometry. 

\, 

A more general setting for Wishart laws is available, given that one can work in a symmetric monoidal category of topological spaces (for the sample space and parameter space). In this very general setting, any SCS cone is isomorphic to a unique product of irreducible SCS cones (see Prop.~\ref{P:Vclass}), each of which is one of the five types described in Table \ref{tab:cones}.

Therefore, one can state the following:
\begin{thm}
Let $W_{\sigma,\xi}$ be a Wishart distribution parametrized by a SCS cone $\Omega$. Let $\xi$ be its multiplier and $\sigma$ its  parameter. Let $I$ be a finite set.
Then, the Wishart distribution $W_{\sigma,\xi}$ can be decomposed as
$$W_{\sigma,\xi}= \bigotimes\limits_{i\in I} W_{\sigma_i,\xi_i},$$  if and only if 
$\Omega$ decomposes into a linear combination of irreducible SCS cones (cf. Table \ref{tab:cones}) $\Omega= \bigotimes\limits_{i\in I} \Omega_i$ such that:
\begin{itemize}
    \item every Wishart distribution $W_{\sigma_i,\xi}$ is parametrized by $\Omega_i$, where $i\in I$;

\item the multiplier on $G(\Omega)$ is $\xi:=\prod\limits_{i\in I} \xi_i$; 

\item  the parameter is given by $\sigma:= (\sigma_i \, |\, i\in I)$.
\end{itemize}
 \end{thm}
\begin{proof}
This comes from the symmetric monoidal structure of CAP and from the construction of SCS cones. 
\end{proof}

The theory of symmetric cones is closely related to the theory of formally real Jordan algebras.

\,

\,

\begin{rem}
   From Wishart law's perspective, the space of real symmetric matrices with prescribed zeros 
(and its subset consisting of positive definite matrices) are important in view of the study of the covariance matrix of a random vector with a given conditional independence.

The difficulty to handle this kind of matrices depends very much on the pattern of zeros, which is expressed by an undirected graph. See works of \cite{LetacMas1,LetacMas2} that give a simple condition of a graph for which the corresponding set of positive definite matrices forms a homogeneous cone. 
 
\end{rem}

We prove the following proposition.
\begin{prop}
~

\begin{enumerate}

\item Assume that there exists an exponential family in $\R^n$ having a variance function being quadratic and homogeneous. Then, there exists a convex symmetric cone for which it forms a family of Wishart laws.
\item Reciprocally, assume that we have a SCS cone defined over $\cC,\R, \oO, \hH$ or the split complex numbers. Then one associates to it its corresponding Jordan algebra and families of Wishart laws, defined for those symmetric cones. These are exponential families invariant under the Lie group preserving the convex symmetric cone. 
 \end{enumerate}
\end{prop}
\begin{proof}

To any symmetric cone one associates one of the 5 types of Jordan Euclidean algebras 
and families of Wishart laws. These are briefly described  by their Laplace transform 
$ \det(1- P(a)(s))^{-p}$  where the determinant and the operator  
$P$ are well defined accordingly to the Jordan algebra. These exponential families of  Wishart have a variance function being homogeneous and quadratic. 

\, 

The first statement is a consequence of the following. By \cite{Let}, the exponential families invariant under the group preserving the symmetric cone lead to considering families of Wishart laws. These can be defined for the symmetric cones over complex, real, split-complex, quaternions, octavic numbers.

\, 

The work of Johnsson--Kotz~\cite{JoKo} studies Wishart laws in the complex case; \cite{A,P} considers Wishart laws in the quaternionic case;.
\cite{Ca1}) investigated Wishart laws in the octavic case and finally \cite{Let} studied Wishart laws in the paracomplex case. 

\, 

A reciprocal statement can be stated. By \cite{Ca1}, if an exponential family in $\R^d$ has a variance function being 
quadratic and homogeneous then there exists a symmetric cone (i.e. a 
Jordan euclidean algebra) for which it is a family of Wishart laws.
Therefore, we have given the proof of the statement.
\end{proof}

\subsection{Conclusion}
We have shown that an important object of the theory of Tomita--Takesaki: the CAH cone (named after Connes--Araki--Haagerup), plays a central role for the Wishart laws, highlighting deeper relations between information geometry (exponential families) and quantum geometry. One aspect of this has been shown using completely different methods in \cite{CoMa,CMM}. This result suggests further emerging relations to quantum geometry to explore.

\end{document}